\documentclass[12pt, a4paper]{article}
\usepackage{amsfonts}
\usepackage{mathrsfs}
\usepackage{latexsym}
\usepackage{xy}
\usepackage{amsfonts,amsmath,amssymb,amsthm}
\usepackage{color}

\xyoption{all}

\newcommand{\bcen}{\begin{center}}     \newcommand{\ecen}{\end{center}}
\newcommand{\bay}{\begin{array}}      \newcommand{\eay}{\end{array}}
\newcommand{\beq}{\begin{eqnarray*}}      \newcommand{\eeq}{\end{eqnarray*}}

\def\diag{\mathrm{diag}}
\def\dim{\mathrm{dim}}
\def\End{\mathrm{End}}

\def\Hom{\mathrm{Hom}}

\def\id{\mathrm{id}}

\def\inj{\mathrm{inj}}

\def\mod{\mathrm{mod}}
\def\Mod{\mathrm{Mod}}

\def\op{\mathrm{op}}
\def\pd{\mathrm{pd}}
\def\per{\mathrm{per}}
\def\proj{\mathrm{proj}}
\def\Proj{\mathrm{Proj}}

\def\RHom{\mathrm{RHom}}

\def\tr{\mathrm{tr}}

\def\tria{\mathrm{tria}}

\begin{document}

\newtheorem{theorem}{Theorem}
\newtheorem{proposition}{Proposition}
\newtheorem{lemma}{Lemma}
\newtheorem{corollary}{Corollary}
\newtheorem{remark}{Remark}
\newtheorem{example}{Example}
\newtheorem{definition}{Definition}
\newtheorem*{conjecture}{Conjecture}
\newtheorem{question}{Question}

\title{\large\bf Reducing homological conjectures by $n$-recollements}

\author{\large Yang Han and Yongyun Qin}

\date{\footnotesize KLMM, ISS, AMSS,
Chinese Academy of Sciences, Beijing 100190, P.R. China.\\ E-mail:
hany@iss.ac.cn ; qinyongyun2006@126.com}

\maketitle

\begin{abstract} $n$-recollements of triangulated categories and
$n$-derived-simple algebras are introduced. The relations between
the $n$-recollements of derived categories of algebras and the
Cartan determinants, homological smoothness and Gorensteinness of
algebras respectively are clarified. As applications, the Cartan
determinant conjecture is reduced to $1$-derived-simple algebras,
and the Gorenstein symmetry conjecture is reduced to
$2$-derived-simple algebras.
\end{abstract}

\medskip


\medskip


\section{\large Introduction}

\indent\indent Throughout $k$ is a fixed field and all algebras are
associative $k$-algebras with identity unless stated otherwise.
Recollements of triangulated categories were introduced by
Beilinson, Bernstein and Deligne \cite{BBD82}, and play an important
role in algebraic geometry and representation theory. Here, we focus
on the recollements of derived categories of algebras which are the
generalization of derived equivalences and provide a useful
reduction technique for some homological properties such as the
finiteness of global dimension \cite{Wie91,Koe91,AKLY13}, the
finiteness of finitistic dimension \cite{Hap93,CX14} and the
finiteness of Hochschild dimension \cite{Han14}, some homological
invariants such as $K$-theory
\cite{TT90,Yao92,Nee04,Sch06,CX12,AKLY13}, Hochschild homology and
cyclic homology \cite{Kel98} and Hochschild cohomology \cite{Han14},
and some homological conjectures such as the finitistic dimension
conjecture \cite{Hap93,CX14} and the Hochschild homology dimension
conjecture \cite{Han06}.

In a recollement, two functors in the first layer always preserve
compactness, i.e., send compact objects to compact ones,
but other functors are not the case in general.
If a recollement is {\it perfect}, i.e.,
two functors in the second layer also preserve compactness,
then the Hochschild homologies, cyclic
homologies and $K$-groups of the middle algebra are the direct sum
of those of outer two algebras respectively
\cite{Kel98,CX12,AKLY13}. Moreover, in this situation, the relations
between recollements and the finitistic dimensions of algebras can
be displayed very completely \cite{CX14}. In order to clarify the
relations between recollements and the homological smoothness and
Gorensteinness of algebras respectively, we need even more layers of
functors preserving compactness, which leads to the
concept of $n$-recollement of triangulated categories inspired by
that of ladder \cite{BGS88}, and further $n$-derived-simple algebra.
In terms of $n$-recollements, the relations between recollements and
the Cartan determinants, homological smoothness and Gorensteinness
of algebras respectively are expressed as follows.

\medskip

{\bf Theorem I.} {\it Let $A$, $B$ and $C$ be finite dimensional
algebras, and $\mathcal{D}(\Mod A)$ admit an $n$-recollement
relative to $\mathcal{D}(\Mod B)$ and $\mathcal{D}(\Mod C)$ with $n
\geq 2$. Then $\det C(A)= \det C(B) \cdot \det C(C).$}

\medskip

{\bf Theorem II.} {\it Let $A$, $B$ and $C$ be algebras and
$\mathcal{D}(\Mod A)$ admit an $n$-recollement relative to
$\mathcal{D}(\Mod B)$ and $\mathcal{D}(\Mod C)$.

{\rm (1) $n=1$:} if $A$ is homologically smooth then so is $B$;

{\rm (2) $n=2$:} if $A$ is homologically smooth then so are $B$ and
$C$;

{\rm (3) $n \geq 3$:}  $A$ is homologically smooth if and only if so
are $B$ and $C$. }

\medskip

{\bf Theorem III.} {\it Let $A$, $B$ and $C$ be finite dimensional
algebras, and \linebreak $\mathcal{D}(\Mod A)$ admit an $n$-recollement
relative to $\mathcal{D}(\Mod B)$ and $\mathcal{D}(\Mod C)$.

{\rm (1) $n=3$:} if $A$ is Gorenstein then so are $B$ and $C$;

{\rm (2) $n \geq 4$:} $A$ is Gorenstein if and only if so are $B$
and $C$. }

\medskip

As applications of Theorem I and Theorem III, we will show that the
Cartan determinant conjecture and the Gorenstein symmetry conjecture
can be reduced to 1-derived-simple algebras and 2-derived-simple
algebras respectively.

The paper is organized as follows: In section 2, we will introduce
the concepts of $n$-recollement of triangulated categories and
$n$-derived-simple algebra, and provide some typical examples and
existence criteria of $n$-recollements of derived categories of
algebras. In section 3, Theorem I is obtained and the Cartan
determinant conjecture is reduced to 1-derived-simple algebras. In
section 4, we will prove Theorem II. In section 5, Theorem III is
shown and the Gorenstein symmetry conjecture is reduced to
2-derived-simple algebras.

\section{\large $n$-recollements and $n$-derived-simple
algebras}

\indent\indent In this section, we will introduce the concepts of
$n$-recollement of triangulated categories and $n$-derived-simple
algebra, and provide some examples and existence criteria of the
$n$-recollements of derived categories of algebras. As we will see,
the language of $n$-recollements is very convenient for us to
observe the relations between recollements and certain homological
properties, especially the Gorensteinness of algebras.

\subsection{$n$-recollements of triangulated categories}

\begin{definition}{\rm (Beilinson-Bernstein-Deligne \cite{BBD82})
Let $\mathcal{T}_1$, $\mathcal{T}$ and $\mathcal{T}_2$ be
triangulated categories. A {\it recollement} of $\mathcal{T}$
relative to $\mathcal{T}_1$ and $\mathcal{T}_2$ is given by
$$\xymatrix@!=4pc{ \mathcal{T}_1 \ar[r]^{i_*=i_!} & \mathcal{T} \ar@<-3ex>[l]_{i^*}
\ar@<+3ex>[l]_{i^!} \ar[r]^{j^!=j^*} & \mathcal{T}_2
\ar@<-3ex>[l]_{j_!} \ar@<+3ex>[l]_{j_*}}$$ such that

(R1) $(i^*,i_*), (i_!,i^!), (j_!,j^!)$ and $(j^*,j_*)$ are adjoint
pairs of triangle functors;

(R2) $i_*$, $j_!$ and $j_*$ are full embeddings;

(R3) $j^!i_*=0$ (and thus also $i^!j_*=0$ and $i^*j_!=0$);

(R4) for each $X \in \mathcal {T}$, there are triangles

$$\begin{array}{l} j_!j^!X \rightarrow X  \rightarrow i_*i^*X  \rightarrow
\\ i_!i^!X \rightarrow X  \rightarrow j_*j^*X  \rightarrow
\end{array}$$ where the arrows to and from $X$ are the counits and the
units of the adjoint pairs respectively. }\end{definition}

\begin{definition}{\rm
Let $\mathcal{T}_1$, $\mathcal{T}$ and $\mathcal{T}_2$ be
triangulated categories, and $n$ a positive integer. An {\it
$n$-recollement} of $\mathcal{T}$ relative to $\mathcal{T}_1$ and
$\mathcal{T}_2$ is given by $n+2$ layers of triangle functors
$$\xymatrix@!=4pc{ \mathcal{T}_1 \ar@<+1ex>[r] \ar@<-3ex>[r]_\vdots & \mathcal{T}
\ar@<+1ex>[r]\ar@<-3ex>[r]_\vdots \ar@<-3ex>[l] \ar@<+1ex>[l] &
\mathcal{T}_2 \ar@<-3ex>[l] \ar@<+1ex>[l]}$$ such that every
consecutive three layers form a recollement. }\end{definition}

Obviously, a 1-recollement is nothing but a recollement. Moreover,
if $\mathcal{T}$ admits an $n$-recollement relative to
$\mathcal{T}_1$ and $\mathcal{T}_2$, then it must admit an
$m$-recollement relative to $\mathcal{T}_1$ and $\mathcal{T}_2$ for
all $1 \leq m \leq n$ and a $p$-recollement relative to $\mathcal{T}_2$ and
$\mathcal{T}_1$ for all $1 \leq p \leq n-1$.

\begin{remark}{\rm Let $\mathcal{T}$ be a skeletally small $k$-linear triangulated
category with finite dimensional Hom-sets and split idempotents. If
$\mathcal{T}$ has a Serre functor and admits a recollement relative
to $\mathcal{T}_1$ and $\mathcal{T}_2$ then it admits an
$n$-recollement relative to $\mathcal{T}_1$ and $\mathcal{T}_2$
(resp. $\mathcal{T}_2$ and $\mathcal{T}_1$) for all $n \in
\mathbb{Z}^+$ by \cite[Theorem 7]{Jor10}.

}\end{remark}

\subsection{$n$-recollements of derived categories of algebras}

\indent\indent Let $A$ be an algebra. Denote by $\Mod A$ the
category of right $A$-modules, and by $\mod A$, $\Proj A$, $\proj A$
and $\inj A$ its full subcategories consisting of all finitely
generated modules, projective modules, finitely generated projective
modules and finitely generated injective modules, respectively.  For
$* \in \{{\rm nothing}, -, +, b \}$ and $\mathcal{A} = \Mod A$ or
any above subcategory of $\Mod A$, denote by $K^*(\mathcal{A})$
(resp. $\mathcal{D}^*(\mathcal{A})$) the homotopy category (resp.
derived category) of cochain complexes of objects in $\mathcal{A}$
satisfying the corresponding boundedness condition. Up to
isomorphism, the objects in $K^{b}(\proj A)$ are precisely all the
compact objects in $\mathcal{D}(\Mod A)$. For convenience, we do not
distinguish $K^{b}(\proj A)$ from the {\it perfect derived category}
$\mathcal{D}_{\per}(A)$ of $A$, i.e., the full triangulated
subcategory of $\mathcal{D} A$ consisting of all compact objects,
which will not cause any confusion. Moreover, we also do not
distinguish $K^b(\inj A), \mathcal{D}^b(\Mod A), \mathcal{D}^b(\mod
A), \mathcal{D}^-(\Mod A)$ and $\mathcal{D}^+(\Mod A)$ from their
essential images under the canonical full embeddings into
$\mathcal{D}(\Mod A)$. Usually, we just write $\mathcal{D} A$
instead of $ \mathcal{D}(\Mod A)$.

In this paper, we focus on the $n$-recollements of derived
categories of algebras, i.e., all three triangulated categories in
an $n$-recollement are the derived categories of algebras. Clearly,
in the $n$-recollement, the upper $n$ layers of functors have right
adjoints preserving direct sums, thus they preserve compactness.

Now we provide some typical examples of $n$-recollements.

\begin{example} \label{Example-n-recollement}
{\rm (1) Stratifying ideals \cite{CPS96}. Let $A$ be an algebra, and
$e$ an idempotent of $A$ such that $AeA$ is a {\it stratifying
ideal}, i.e., $Ae \otimes^L_{eAe} eA \cong AeA$ canonically. Then
$\mathcal{D}A$ admits a 1-recollement relative to
$\mathcal{D}(A/AeA)$ and $\mathcal{D}(eAe)$.

(2) Triangular matrix algebras \cite[Example 3.4]{AKLY13}. Let $B$
and $C$ be algebras, $M$ a $C$-$B$-bimodule, and $A =
\left[\begin{array}{cc} B & 0 \\ M & C \end{array}\right]$. Then
$\mathcal{D}A$ admits a 2-recollement relative to $\mathcal{D}B$ and
$\mathcal{D}C$. Furthermore, $\mathcal{D}A$ admits a 3-recollement
relative to $\mathcal{D}C$ and $\mathcal{D}B$ if $_CM\in K^b(\proj
C^{\op})$, and $\mathcal{D}A$ admits a 3-recollement relative to
$\mathcal{D}B$ and $\mathcal{D}C$ if $M_B\in K^b(\proj B)$. What is
more, $\mathcal{D}A$ admits a 4-recollement relative to
$\mathcal{D}C$ and $\mathcal{D}B$ if $_CM \in K^b(\proj C^{\op})$
and $M_B \in K^b(\proj B)$. Note that the algebras $A,B$ and $C$
here need not be finite dimensional.

(3) Let $A$ be a finite dimensional algebra of finite global dimension and
$\mathcal{D}A$ admit a recollement relative to $\mathcal{D}B$ and
$\mathcal{D}C$. Then this recollement can be extended to an
$n$-recollement for all $n \in \mathbb{Z}^+$ (Ref. \cite[Proposition
3.3]{AKLY13}).

(4) A derived equivalence induces a {\it trivial} $n$-recollement,
i.e., an $n$-recollement whose left term or right term is zero, for
all $n \in \mathbb{Z}^+$. }\end{example}

Usually we pay more attention to the $n$-recollements of derived
categories of finite dimensional algebras. In this situation, we
have some useful existence criteria of $n$-recollements.

\begin{lemma} \label{Lemma-adjoint}
Let $A$ and $B$ be finite dimensional algebras, and the triangle
functor $F : \mathcal{D}A \rightarrow \mathcal{D}B$ left adjoint to
$G : \mathcal{D}B \rightarrow \mathcal{D}A$. Then:

{\rm (1)} $F$ restricts to $K^b(\proj)$ if and only if $G$ restricts
to $\mathcal{D}^b(\mod)$;

{\rm (2)}  $F$ restricts to $\mathcal{D}^b(\mod)$ if and only if $G$
restricts to $K^b(\inj)$.
\end{lemma}

\begin{proof}
(1) follows from \cite[Lemma 2.7]{AKLY13}, and (2) is dual to (1).
\end{proof}

\begin{lemma}\label{Lemma-restrict}
Let $A$, $B$ and $C$ be finite dimensional algebras, and
$$\xymatrix@!=4pc{\mathcal{D}B \ar[r]^{i_*} & \mathcal{D}A \ar@<-3ex>[l]_{i^*}
\ar@<+3ex>[l]_{i^!} \ar[r]^{j^!} & \mathcal{D}C \ar@<-3ex>[l]_{j_!}
\ar@<+3ex>[l]_{j_*} }$$ a recollement. Then the following statements
hold:

{\rm (1)} $i^*$ and $j_!$ restrict to $K^b(\proj)$;

{\rm (2)}  $i_*$ and $j^!$ restrict to $\mathcal{D}^b(\mod)$;

{\rm (3)} $i^!$ and $j_*$ restrict to $K^b(\inj)$.
\end{lemma}

\begin{proof}
(1) is clear. (2) and (3) follow from Lemma~\ref{Lemma-adjoint}.
\end{proof}

\begin{lemma} \label{Lemma-downwards}
Let $A$, $B$ and $C$ be finite dimensional algebras, and
$$\xymatrix@!=4pc{\mathcal{D}B \ar[r]^{i_*} & \mathcal{D}A \ar@<-3ex>[l]_{i^*}
\ar@<+3ex>[l]_{i^!} \ar[r]^{j^!} & \mathcal{D}C \ar@<-3ex>[l]_{j_!}
\ar@<+3ex>[l]_{j_*} } \eqno {\rm (R)}$$  a recollement. Then the
following statements are equivalent:

{\rm (1)} The recollement {\rm (R)} can be extended one-step
downwards;

{\rm (2)} $i_*$ or/and $j^!$ restricts to $K^b(\proj)$;

{\rm (2')}  $i_*B \in K^b(\proj A)$ or/and $j^!(A) \in K^b(\proj
C)$;

{\rm (3)} $i^!$ or/and $j_*$ restricts to $\mathcal{D}^b(\mod)$;

{\rm (4)}  The recollement {\rm (R)} restricts to
$\mathcal{D}^-(\Mod)$.
\end{lemma}

\begin{proof} (1) $\Leftrightarrow$ (2): It follows from \cite[Proposition 3.2
(a)]{AKLY13}.

(2) $\Leftrightarrow$ (2'): It follows from \cite[Lemma
2.5]{AKLY13}.

(2) $\Leftrightarrow$ (3): It follows from
Lemma~\ref{Lemma-adjoint}.

(4) $\Leftrightarrow$ (2'): It follows from \cite[Proposition
4.11]{AKLY13}.
\end{proof}

\begin{lemma} \label{Lemma-upwards}
Let $A$, $B$ and $C$ be finite dimensional algebras, and
$$\xymatrix@!=4pc{\mathcal{D}B \ar[r]^{i_*} & \mathcal{D}A \ar@<-3ex>[l]_{i^*}
\ar@<+3ex>[l]_{i^!} \ar[r]^{j^!} & \mathcal{D}C \ar@<-3ex>[l]_{j_!}
\ar@<+3ex>[l]_{j_*} } \eqno {\rm (R)}$$  a recollement. Then the
following statements are equivalent:

{\rm (1)} The recollement {\rm (R)} can be extended one-step
upwards;

{\rm (2)} $i_*$ or/and $j^!$ restricts to $K^b(\inj)$;

{\rm (2')}  $i_*(DB) \in K^b(\inj A)$ or/and $j^!(DA) \in K^b(\inj
C)$ where $D = \Hom_k(-,k)$;

{\rm (3)} $i^*$ or/and $j_!$ restricts to $\mathcal{D}^b(\mod)$;

{\rm (4)}  The recollement {\rm (R)} restricts to
$\mathcal{D}^+(\Mod)$.
\end{lemma}

\begin{proof} This lemma is dual to Lemma~\ref{Lemma-downwards}. \end{proof}

\begin{proposition} \label{Proposition-2-recollement}
Let $A$, $B$ and $C$ be finite dimensional algebras. Then the
following conditions are equivalent:

{\rm (1)} $\mathcal{D}A$ admits a $2$-recollement relative
to $\mathcal{D}B$ and $\mathcal{D}C$;

{\rm (2)} $\mathcal{D}A$ admits a recollement relative to
$\mathcal{D}B$ and $\mathcal{D}C$ in which two functors in the
second layer restrict to $K^b(\proj)$;

{\rm (3)} $\mathcal{D}A$ admits a recollement relative to
$\mathcal{D}B$ and $\mathcal{D}C$ in which two functors in the third
layer restrict to $\mathcal{D}^b(\mod)$;

{\rm (4)} $\mathcal{D}^-(\Mod A)$ admits a recollement relative to
$\mathcal{D}^-(\Mod B)$ and $\mathcal{D}^-(\Mod C)$;

{\rm (5)} $\mathcal{D}A$ admits a recollement relative to
$\mathcal{D}C$ and $\mathcal{D}B$ in which two functors in the first
layer restrict to $\mathcal{D}^b(\mod)$;

{\rm (6)} $\mathcal{D}A$ admits a recollement relative to
$\mathcal{D}C$ and $\mathcal{D}B$ in which two functors in the
second layer restrict to $K^b(\inj)$;

{\rm (7)} $\mathcal{D}^+(\Mod A)$ admits a recollement relative to
$\mathcal{D}^+(\Mod C)$ and $\mathcal{D}^+(\Mod B)$.

\end{proposition}

\begin{proof} (1) $\Leftrightarrow$ (2) $\Leftrightarrow$ (3) $\Leftrightarrow$ (4):
By \cite[Proposition 4.1]{AKLY13}, any
$\mathcal{D}^-(\Mod)$-recollement can be lifted to a
$\mathcal{D}(\Mod)$-recollement. Then it follows from
Lemma~\ref{Lemma-downwards}.

(1) $\Leftrightarrow$ (5) $\Leftrightarrow$ (6) $\Leftrightarrow$
(7): Analogous to \cite[Proposition 4.1]{AKLY13}, any
$\mathcal{D}^+(\Mod)$-recollement can be lifted to a
$\mathcal{D}(\Mod)$-recollement as well. Then it follows from
Lemma~\ref{Lemma-upwards}.
\end{proof}

\begin{proposition} \label{Proposition-3-recollement}
Let $A$, $B$ and $C$ be finite dimensional algebras. Then the
following conditions are equivalent:

{\rm (1)} $\mathcal{D}A$ admits a $3$-recollement relative
to $\mathcal{D}B$ and $\mathcal{D}C$;

{\rm (2)} $\mathcal{D}A$ admits a recollement relative to
$\mathcal{D}B$ and $\mathcal{D}C$ in which all functors restrict to
$K^b(\proj)$;

{\rm (3)} $\mathcal{D}^b(\mod A)$ admits a recollement relative to
$\mathcal{D}^b(\mod C)$ and $\mathcal{D}^b(\mod B)$;

{\rm (4)} $\mathcal{D}A$ admits a recollement relative to
$\mathcal{D}B$ and $\mathcal{D}C$ in which all functors restrict to
$K^b(\inj)$;

{\rm (5)} $\mathcal{D}^b(\Mod A)$ admits a recollement relative to
$\mathcal{D}^b(\Mod C)$ and $\mathcal{D}^b(\Mod B)$.
\end{proposition}

\begin{proof}
(1) $\Leftrightarrow$ (2) $\Leftrightarrow$ (3) : It follows from
\cite[Proposition 4.1]{AKLY13} and Lemma~\ref{Lemma-downwards}.

(1) $\Leftrightarrow$ (3) $\Leftrightarrow$ (4) : It follows from
\cite[Proposition 4.1]{AKLY13} and Lemma~\ref{Lemma-upwards}.

(3) $\Leftrightarrow$ (5) : It follows from \cite[Proposition 4.1
and Corollary 4.9]{AKLY13}.
\end{proof}

\subsection{$n$-derived-simple algebras}

\indent\indent For any recollement of derived categories of finite
dimensional algebras, the Grothendieck group of the middle algebra
is the direct sum of those of the outer two algebras
\cite[Proposition 6.5]{AKLY13}. Thus the process of reducing
homological properties, homological invariants and homological
conjectures by recollements must terminate after finitely many
steps. This leads to derived simple algebras, whose derived
categories admit no nontrivial recollements any more. This
definition dates from Wiedemann \cite{Wie91}, where the author
considered the stratifications of bounded derived categories. Later
on, recollements of unbounded and bounded above derived categories
attract considerable attention, and so do the corresponding derived
simple algebras \cite{AKLY13}. When we consider the stratifications
along $n$-recollements, $n$-derived-simple algebras are defined
naturally.

\begin{definition}{\rm
A finite dimensional algebra $A$ is said to be {\it
$n$-derived-simple} if its derived category $\mathcal{D}A$ admits no
nontrivial $n$-recollements. }\end{definition}

Clearly, an $n$-derived-simple algebra must be
indecomposable/connected. Note that 1-derived-simple algebras are
just the $\mathcal{D}({\rm Mod})$-derived simple algebras. For
finite dimensional algebras, by
Proposition~\ref{Proposition-2-recollement} and
Proposition~\ref{Proposition-3-recollement}, 2 (resp.
3)-derived-simple algebras are exactly $\mathcal{D}^-({\rm Mod})$
(resp. $\mathcal{D}^b(\mod)$)-derived simple algebras in the sense
of \cite{AKLY13}. Moreover, $n$-derived-simple algebras must be
$m$-derived simple for all $m \geq n$, and it is worth noting that
for a finite dimensional algebra $A$ of finite global dimension, the
$n$-derived-simplicity of $A$ does not depend on the choice of $n$.

Although it is difficult to find out all the $n$-derived-simple
algebras, there are still some known examples.

\begin{example}\label{Example-n-derived-simple}
{\rm (1) Finite dimensional local algebras, blocks of finite group
algebras and indecomposable representation-finite symmetric algebras
are 1-derived-simple \cite{Wie91,LY12};

(2) Some finite dimensional two-point algebras of finite global
dimension are $n$-derived-simple for all $n \in \mathbb{Z}^+$ (Ref.
\cite{Happ91,LY13});

(3) Indecomposable symmetric algebras are 2-derived-simple
\cite{LY12};

(4) There exist 2-derived-simple algebras which are not
1-derived-simple \cite[Example 5.8]{AKLY13}, 3-derived-simple
algebras which are not 2-derived-simple \cite[Example 5.10]{AKLY13},
and 4-derived-simple algebras which are not 3-derived-simple
\cite[Example 4.13]{AKLY13}, respectively. }
\end{example}

Let's end this section by listing some known results on reducing
homological conjectures via recollements. First, the {\it finitistic
dimension conjecture}, which says that every finite dimensional
algebra has finite finitistic dimension, was reduced to
3-derived-simple algebras by Happel \cite{Hap93}. Recently, Chen and
Xi extended his result by reducing the finitistic dimension
conjecture to 2-derived-simple algebras \cite{CX14}. Second, it
follows from \cite[Proposition 2.9(b)]{Kel98} and \cite[Proposition
2.14]{AKLY13} that the {\it Hochschild homology dimension
conjecture}, which states that the finite dimensional algebras of
finite Hochschild homology dimension are of finite global dimension
\cite{Han06}, can be reduced to 2-derived-simple algebras. Last but
not least, both {\it vanishing conjecture} and {\it dual vanishing
conjecture} can be reduced to 3-derived-simple algebras
\cite{Yang14}.

\section{\large $n$-recollements and Cartan determinants}

\indent\indent In this section, we will observe the relations
between $n$-recollements and the Cartan determinants of algebras,
and reduce the Cartan determinant conjecture to 1-derived-simple
algebras.

Let $\mathcal{E}$ be a skeletally small exact category, $F$ the free
abelian group generated by the isomorphism classes $[X]$ of objects
$X$ in $\mathcal{E}$, and $F_0$ be the subgroup of $F$ generated by
$[X]-[Y]+[Z]$ for all conflations $0 \rightarrow X \rightarrow Y
\rightarrow Z \rightarrow 0 $ in $\mathcal{E}$. The {\it
Grothendieck group} $K_0(\mathcal{E})$ of $\mathcal{E}$ is the
factor group $F/F_0$. The Grothendieck group of a skeletally small
triangulated category is defined similarly, just need replace
conflations with triangles.

Let $A$ be a finite dimensional algebra and $\{P_1, \cdots , P_r\}$
a complete set of non-isomorphic indecomposable projective
$A$-modules. Then their tops $\{S_1, \cdots , S_r\}$ form a complete
set of non-isomorphic simple $A$-modules. The map $C_A : K_0(\proj
A) \rightarrow K_0(\mod A),$ $[P] \mapsto [P]$, is called the {\it
Cartan map} of $A$, which can be extended to $C_A : K_0(K^b(\proj
A)) \rightarrow K_0(\mathcal{D}^b(\mod A)), [X] \mapsto [X]$. The
matrix of the Cartan map $C_A$ under the $\mathbb{Z}$-basis
$\{[P_1], \cdots , [P_r]\}$ of $K_0(\proj A)$ and the
$\mathbb{Z}$-basis $\{[S_1], \cdots , [S_r]\}$ of $K_0(\mod A)$ is
called the {\it Cartan matrix} of $A$, and denoted by $C(A)$.
Namely, $C(A)$ is the $r \times r$ matrix whose $(i,j)$-th entry
$c_{ij}$ is the multiplicity of $S_i$ in $P_j$. Obviously, $c_{ij}$
equals to the composition length of the $\End_A(P_i)$-module
$\Hom_A(P_i,P_j)$, or $\dim_k\Hom_A(P_i,P_j) / \dim_k \End_A(S_i)$.

Now we study the relation between $n$-recollements and the Cartan
determinant of algebras. For convenience, we define $\det C(0) = 1$.
The following theorem is just Theorem I.

\begin{theorem} \label{Theorem-Cartan determinant}
Let $A'$, $A$ and $A''$ be finite dimensional algebras, and
$\mathcal{D}A$ admit an $n$-recollement relative to $\mathcal{D}A'$
and $\mathcal{D}A''$ with $n \geq 2$. Then $\det C(A)= \det C(A')
\cdot \det C(A'')$.
\end{theorem}

\begin{proof}

It follows from Proposition~\ref{Proposition-2-recollement} and
Lemma~\ref{Lemma-restrict} that $\mathcal{D}A$ admits a recollement
$$\xymatrix@!=4pc{ \mathcal{D}A' \ar[r]^{i_*}
& \mathcal{D}A \ar@<-3ex>[l]_{i^*} \ar@<+3ex>[l]_{i^!} \ar[r]^{j^!}
& \mathcal{D}A'' \ar@<-3ex>[l]_{j_!} \ar@<+3ex>[l]_{j_*}}$$ such
that $i^*, i_*, j_!$ and $j^!$ restrict to $K^b(\proj)$, and $i_*,
i^!, j^!$ and $j_*$ restrict to $\mathcal{D}^b(\mod)$.

Let $\{P'_1, \cdots , P'_{r'} \}$ (resp. $\{P_1, \cdots , P_r \}$,
$\{P''_1, \cdots , P''_{r''}\}$) be a complete set of non-isomorphic
indecomposable projective $A'$-modules (resp. $A$-modules,
$A''$-modules). Then their tops $\{S'_1, \cdots , S'_{r'} \}$ (resp.
$\{S_1, \cdots , S_r \}$, $\{S''_1, \cdots , S''_{r''}\}$) form a
complete set of non-isomorphic simple $A'$-modules (resp.
$A$-modules, $A''$-modules).  By \cite[Theorem 1.1]{CX12} or
\cite[Proposition 6.5]{AKLY13}, we have $r'+r''=r$.

Consider the triangles $j_!j^!P_u \rightarrow P_u \rightarrow
i_*i^*P_u \rightarrow $ for all $1 \leq u \leq r$. Since $P_u \in
K^b(\proj A)$, we have $j^!P_u \in K^b(\proj A'') = \tria \{P''_1,
\cdots , P''_{r''}\} \subseteq \mathcal {D}A''$ and $i^*P_u \in
K^b(\proj A') = \tria \{P'_1, \cdots , P'_{r'}\} \subseteq \mathcal
{D}A'$. Here, for a class $\mathcal {X}$ of objects in a
triangulated category $\mathcal {T}$, $\tria \mathcal {X}$ denotes
the smallest full triangulated subcategory of $\mathcal {T}$
containing $\mathcal {X}$. Furthermore, we have $j_!j^!P_u \in \tria
\{j_!P''_1, \cdots , j_!P''_{r''}\} \subseteq \mathcal {D}A$ and
$i_*i^*P_u \in \tria \{i_*P'_1, \cdots , i_*P'_{r'}\} \subseteq
\mathcal {D}A$. Hence $P_u \in \tria \{i_*P'_1, \cdots , i_*P'_{r'},
j_!P''_1, \cdots , j_!P''_{r''}\} \subseteq \mathcal {D}A$, and
$K^b(\proj A) = $ \linebreak $\tria \{P_1, \cdots , P_r\} = \tria
\{i_*P'_1, \cdots , i_*P'_{r'}, j_!P''_1, \cdots , j_!P''_{r''}\}
\subseteq \mathcal {D}A$. Therefore, $\{[i_*P'_1], \cdots ,
[i_*P'_{r'}], [j_!P''_1], \cdots , [j_!P''_{r''}]\}$ is a
$\mathbb{Z}$-basis of $K_0(K^b(\proj A))$.

Consider the triangles $i_*i^!S_u \rightarrow S_u \rightarrow
j_*j^!S_u \rightarrow $ for all $1 \leq u \leq r$. Since $S_u \in
\mathcal {D}^b(\mod A)$, we have $i^!S_u \in \mathcal {D}^b(\mod A')
= \tria\{S'_1, \cdots , S'_{r'}\} \subseteq \mathcal {D}A'$ and
$j^!S_u \in \mathcal {D}^b(\mod A'') = \tria\{S''_1, \cdots ,
S''_{r''}\} \subseteq \mathcal {D}A''$. Furthermore, we have
$i_*i^!S_u \in \tria \{ i_*S'_1, \cdots , i_*S'_{r'}\} \subseteq
\mathcal {D}A$ and $j_*j^!S_u \in \tria \{ j_*S''_1, \cdots ,
j_*S''_{r''}\} \subseteq \mathcal {D}A$. Hence $S_u \in \tria
\{i_*S'_1, \cdots , i_*S'_{r'}, j_*S''_1, \cdots , j_*S''_{r''}\}
\subseteq \mathcal {D}A$, and $\mathcal {D}^b(\mod A)$ $=
\tria\{S_1, \cdots , S_r\} = \tria \{i_*S'_1, \cdots , i_*S'_{r'},
j_*S''_1, \cdots , j_*S''_{r''}\} \subseteq \mathcal {D}A$.
Therefore, $\{[i_*S'_1], \cdots , [i_*S'_{r'}], [j_*S''_1], \cdots
,$ $[j_*S''_{r''}]\}$ is a $\mathbb{Z}$-basis of $K_0(\mathcal
{D}^b(\mod A))$.

We have clearly the following commutative diagram $$\xymatrix{
K_0(K^b(\proj A')) \ar[d]^-{C_{A'}} \ar[r]^-{i_*} & K_0(K^b(\proj
A)) \ar[d]^-{C_{A}}
\ar[r]^-{j^!} & K_0(K^b(\proj A'')) \ar[d]^-{C_{A''}} \\
K_0(\mathcal {D}^b(\mod A')) \ar[r]^-{i_*} & K_0(\mathcal {D}^b(\mod
A)) \ar[r]^-{j^!} & K_0(\mathcal {D}^b(\mod A'')), }$$ where the
horizonal maps are naturally induced by the functors $i_*$ and
$j^!$. It is not difficult to see that the matrix of the Cartan map
$C_A$ of $A$ under the $\mathbb{Z}$-basis $\{[i_*P'_1], \cdots ,
[i_*P'_{r'}], [j_!P''_1], \cdots , [j_!P''_{r''}]\}$ of
$K_0(K^b(\proj A))$ and the $\mathbb{Z}$-basis $\{[i_*S'_1], \cdots
, [i_*S'_{r'}], [j_*S''_1], \cdots , [j_*S''_{r''}]\}$ of
$K_0(\mathcal {D}^b(\mod A))$ is of the form $\left[\begin{array}{cc} C(A') & \ast \\
0 & C(A'') \end{array}\right]$. Note that the matrix of the Cartan
map $C_A$ of $A$ under the $\mathbb{Z}$-basis $\{[P_1], \cdots ,
[P_r]\}$ of $K_0(K^b(\proj A))$ and the $\mathbb{Z}$-basis $\{[S_1],
\cdots , [S_r]\}$ of $K_0(\mathcal {D}^b(\mod A))$ is just $C(A)$.

Both $\{[i_*P'_1], \cdots , [i_*P'_{r'}], [j_!P''_1], \cdots ,
[j_!P''_{r''}]\}$ and $\{[P_1], \cdots , [P_r]\}$ are
$\mathbb{Z}$-bases of $K_0(K^b(\proj A))$, thus there exist
invertible matrices $U$ and $V$ in $M_r(\mathbb{Z})$ such that
$$([i_*P'_1], \cdots , [i_*P'_{r'}], [j_!P''_1], \cdots ,
[j_!P''_{r''}]) = ([P_1], \cdots , [P_r]) \cdot U$$ and
$$([P_1], \cdots , [P_r]) = ([i_*P'_1], \cdots , [i_*P'_{r'}],
[j_!P''_1], \cdots , [j_!P''_{r''}]) \cdot V.$$ Hence, $UV=VU=I$,
the identity matrix. Therefore, $\det U = \det V = \pm 1$.

Both $\{[i_*S'_1], \cdots , [i_*S'_{r'}], [j_*S''_1], \cdots ,
[j_*S''_{r''}]\}$ and $\{[S_1], \cdots , [S_r]\}$ are
$\mathbb{Z}$-bases of $K_0(\mathcal {D}^b(\mod A))$, thus there
exist invertible matrices $Q$ and $R$ in $M_r(\mathbb{Z})$ such that
$$([i_*S'_1], \cdots , [i_*S'_{r'}], [j_*S''_1], \cdots ,
[j_*S''_{r''}]) = ([S_1], \cdots , [S_r]) \cdot Q$$ and
$$([S_1], \cdots , [S_r]) = ([i_*S'_1], \cdots , [i_*S'_{r'}], [j_*S''_1], \cdots ,
[j_*S''_{r''}]) \cdot R.$$ Hence, $QR=RQ=I$, the identity matrix.
Therefore, $\det Q = \det R = \pm 1$.

Combining the equalities above with $([P_1], \cdots , [P_r]) =
([S_1], \cdots , [S_r]) \cdot C(A)$ and $([i_*P'_1], \cdots ,
[i_*P'_{r'}], [j_!P''_1], \cdots , [j_!P''_{r''}]) = ([i_*S'_1],
\cdots , [i_*S'_{r'}], [j_*S''_1], $ \linebreak $
\cdots , [j_*S''_{r''}]) \cdot \left[\begin{array}{cc} C(A') & \ast \\
0 & C(A'') \end{array}\right]$, we have
$C(A) \cdot U = Q \cdot \left[\begin{array}{cc} C(A') & \ast \\
0 & C(A'') \end{array}\right]$. Furthermore, $\det C(A) = \pm \det
C(A') \cdot \det C(A'')$.

On the other hand, we can define a $\mathbb{Z}$-bilinear form
$$\langle -, - \rangle : K_0(K^b(\proj A)) \times K_0(K^b(\proj
A)) \rightarrow \mathbb{Z}$$ by $$\langle [X], [Y] \rangle :=
\sum_{l \in \mathbb{Z}} (-1)^l \ \dim_k \Hom_{K^b(\proj A)}
(X,Y[l]),$$ for all $X, Y \in K^b(\proj A)$.

Since $i_*$ and $j_!$ are full embeddings and $j^!i_*=0$, we have
$$\begin{array} {llll} \langle [i_*P'_u] , [i_*P'_v] \rangle
& = &  \dim _k \Hom_{A'}(P'_u,P'_v), &  u,v = 1, \cdots, r'; \\
\langle [j_!P''_u] , [i_*P'_v] \rangle & = &
0, & u = 1, \cdots , r''; \ v = 1, \cdots , r';\\
\langle [j_!P''_u] , [j_!P''_v]) \rangle & = & \dim_k
\Hom_{A''}(P''_u,P''_v), & u,v = 1, \cdots , r''.
\end{array}$$ Thus the matrix of $\langle - , - \rangle$ under the basis
$\{[i_*P'_1], \cdots , [i_*P'_{r'}], [j_!P''_1], \cdots
, $ \linebreak $[j_!P''_{r''}]\}$ is $\left[\begin{array}{cc} D' \cdot C(A') & \ast \\
0 & D'' \cdot C(A'') \end{array}\right]$ where $D' = \diag\{c'_1,
\cdots , c'_{r'}\}$ with $c'_v = \dim_k\End_{A'}(S'_v)$ for all $v =
1, \cdots , r'$ and $D'' = \diag\{c''_1, \cdots , c''_{r''}\}$ with
$c''_w = \dim_k\End_{A''}(S''_w)$ for all $w = 1, \cdots , r''$.

Let $D = \diag\{c_1, \cdots , c_r\}$ with $c_u = \dim_k\End_A(S_u)$
for all $u = 1, \cdots , r$. Note that the matrix of $\langle -, -
\rangle$ under the basis $\{[P_1], \cdots, [P_r]\}$ is $D \cdot
C(A)$. Thus $D \cdot C(A)$
and $\left[\begin{array}{cc} D' \cdot C(A') & \ast \\
0 & D'' \cdot C(A'') \end{array}\right]$ are the matrices of
$\langle -, - \rangle$ with respect to two different bases. Hence,
there exists an invertible matrix $T \in M_r(\mathbb{Z})$ such that
$D \cdot C(A) = T^{\tr} \cdot \left[\begin{array}{cc} D' \cdot C(A')  & \ast \\
0 & D'' \cdot C(A'') \end{array}\right] \cdot T$. Therefore, $\det
C(A) = \frac{c'_1 \cdots c'_{r'} c''_1 \cdots c''_{r''}}{c_1 \cdots
c_r} \cdot (\det T)^2 \cdot \det C(A') \cdot \det C(A'')$.

If $\det C(A) = 0$ then $\det C(A') \cdot \det C(A'') = 0$. Thus
$\det C(A) = \det C(A') \cdot \det C(A'')$. If $\det C(A) \neq 0$
then $\det C(A') \cdot \det C(A'') \neq 0$. Thus we must have
$\frac{c'_1 \cdots c'_{r'} c''_1 \cdots c''_{r''}}{c_1 \cdots c_r}
\cdot (\det T)^2 =1$ but not $-1$. Hence $\det C(A) = \det C(A')
\cdot \det C(A'')$.
\end{proof}

Applying Theorem~\ref{Theorem-Cartan determinant} to the trivial
2-recollement in Example~\ref{Proposition-2-recollement} (4), we can
obtain the following corollary which generalizes \cite[Proposition
1.5]{BS05} to an arbitrary base field.

\begin{corollary} \label{Corollary-DerEquCartan}
Let $A$ and $B$ be derived equivalent finite
dimensional algebras. Then $\det C(A) = \det C(B)$. \end{corollary}

Next we study the Cartan determinant conjecture. In 1954, Eilenberg
showed that if $A$ is a finite dimensional algebra of finite global
dimension then $\det C(A) = \pm 1$ (Ref. \cite{Eil54}). After that,
the following conjecture was posed:

\medskip

\noindent{\bf Cartan determinant conjecture.} Let $A$ be an artin
algebra of finite global dimension. Then $\det C(A) = 1$.

\medskip

The Cartan determinant conjecture remains open except for some
special classes of algebras, such as the algebras of global
dimension two \cite{Zac83}, the positively graded algebras
\cite{Wil83}, the Cartan filtered algebras \cite{FH86}, the left
serial algebras \cite{BFVH85}, the quasi-hereditary algebras
\cite{BF89}, and the artin algebras admitting a strongly adequate
grading by an aperiodic commutative monoid \cite{Sao98}.

\begin{proposition} \label{Proposition-CarDetConj}
Let $A'$, $A$ and $A''$ be finite dimensional algebras, and
$\mathcal{D}A$ admit a recollement relative to $\mathcal{D}A'$ and
$\mathcal{D}A''$. If both $A'$ and $A''$ satisfy the Cartan
determinant conjecture, then so does $A$.
\end{proposition}

\begin{proof} If $A$ is of finite global dimension then so are
$A'$ and $A''$ by \cite[Proposition 2.14]{AKLY13}. Thus $\det C(A')
= \det C(A'') = 1$ by the assumption and the recollement induces a
2-recollement, see Example~\ref{Example-n-recollement} (3). By
Theorem~\ref{Theorem-Cartan determinant}, we have $\det C(A) = \det
C(A') \cdot \det C(A'') = 1$.
\end{proof}

The following corollary implies that the Cartan determinant
conjecture can be reduced to an arbitrary complete set of
representatives of the derived equivalence classes of finite
dimensional algebras.

\begin{corollary} Let $A$ and $B$ be derived equivalent finite
dimensional algebras. Then $A$ satisfies the Cartan determinant
conjecture if and only if so does $B$. \end{corollary}

\begin{proof} It is enough to apply
Proposition~\ref{Proposition-CarDetConj} to the trivial
recollements, see Example~\ref{Example-n-recollement} (4).
\end{proof}

Applying Proposition~\ref{Proposition-CarDetConj}, we can reduce the
Cartan determinant conjecture to 1-derived-simple algebras.

\begin{corollary} \label{Corollary-Cartan determinant}
The Cartan determinant conjecture holds for all finite dimensional
algebras if and only if it holds for all 1-derived-simple algebras.
\end{corollary}

\begin{proof} For any finite dimensional algebra $A$,
by \cite[Proposition 6.5]{AKLY13}, $\mathcal{D}A$ admits a finite
stratification of derived categories along recollements with
1-derived-simple factors. Then the corollary follows from
Proposition~\ref{Proposition-CarDetConj}.
\end{proof}

Although Corollary~\ref{Corollary-Cartan determinant} provides a
reduction technique, the Cartan determinant conjecture seems far
from being settled, because it is still a problem to deal with all
the 1-derived-simple algebras of finite global dimension.
Nonetheless, for the known examples described in
Example~\ref{Example-n-derived-simple} (1) and (2), the Cartan
determinant conjecture holds true \cite{Happ91,LY13}.

Let's end this section by pointing out that
Theorem~\ref{Theorem-Cartan determinant} can be applied to prove the
$n$-derived-simplicity of certain algebras as well.

\begin{remark} \label{two-point}
{\rm A finite dimensional two-point algebra $A$ with $\det C(A)\leq
0$ must be 2-derived-simple: Otherwise, $\mathcal{D}A$ admits a
non-trivial 2-recollement relative to $\mathcal{D}B$ and
$\mathcal{D}C$. Then both $B$ and $C$ are finite dimensional local
algebras since $A$ has only two simple modules up to isomorphism.
Therefore, $\det C(B) > 0$ and $\det C(C) > 0$. By
Theorem~\ref{Theorem-Cartan determinant}, we get $\det C(A)> 0$. It
is a contradiction. The examples of this kind of 2-derived-simple
algebras include:

(1) $\xymatrix{1 \ar@<+0.5ex>[r]^\alpha & 2 \ar@<+0.5ex>[l]^\beta},
(\alpha \beta)^n=0=(\beta \alpha)^n, n \in \mathbb{Z}^+$ ;

(2) $\xymatrix{1 \ar@<+0.5ex>[r]^\alpha \ar@(ul,dl)_\gamma& 2
\ar@<+0.5ex>[l]^\beta \ar@(ur,dr)^\delta}$, $\alpha \beta = \beta
\alpha = \gamma^2 = \delta^2 = \gamma\alpha-\alpha\delta = \delta\beta-\beta\gamma =0$ ;

(3) Let $A$ be one of the algebras in (1) and (2), and $B$ an
arbitrary finite dimensional local algebra. Then the tensor product
algebra $A \otimes_k B$ is again 2-derived-simple by the same
reason. }
\end{remark}

\begin{remark}
{\rm A representation-finite selfinjective two-point algebra $A$
with $\det C(A) \leq 0$ must be 1-derived-simple. Indeed, for a
representation-finite selfinjective algebra $A$, if $\mathcal{D}A$
admits a recollement relative to $\mathcal{D}B$ and $\mathcal{D}C$,
this recollement must be perfect \cite[Proposition 4.1]{LY12}.
Therefore, the 2-derived-simplicity of these algebras implies the
1-derived-simplicity. For example, the algebras in
Remark~\ref{two-point} (1) are 1-derived-simple.}
\end{remark}

\section{\large $n$-recollements and homological smoothness}

\indent\indent In this section, we will observe the relation between
$n$-recollements and the homological smoothness of algebras.

Let $A$ be an algebra and $A^e := A^{\op} \otimes_k A$ its
enveloping algebra. The algebra $A$ is said to be {\it smooth} if
the projective dimension of $A$ as an $A^e$-module is finite
\cite{VdB98}. The algebra $A$ is said to be {\it homologically
smooth} if $A$ is compact in $\mathcal{D}(A^e)$, i.e., $A$ is
isomorphic in $\mathcal{D}(A^e)$ to an object in $K^b(\proj A^e)$
(Ref. \cite{KS06}). Clearly, all homologically smooth algebras are
smooth. Moreover, if $A$ is a finite dimensional algebra then the
concepts of smoothness and homological smoothness coincide. However,
they are different in general. For example, the infinite Kronecker
algebra is smooth but not homologically smooth \cite[Remark
4]{Han14}.

Let $A$ and $B$ be two derived equivalent algebras. Then, by
\cite[Proposition 2.5]{Ric91}, there is a triangle equivalence
functor from $\mathcal{D}(A^e)$ to $\mathcal{D}(B^e)$ sending
$A_{A^e}$ to $B_{B^e}$. Since the equivalence functor can restrict
to $K^b(\Proj)$ and $K^b(\proj)$, both the smoothness and the
homological smoothness of algebras are invariant under derived
equivalences. Moreover, the relations between recollements and the
smoothness of algebras have been clarified in \cite{Han14}:

\begin{proposition} \label{Proposition-smooth} {\rm (See \cite[Theorem 3]{Han14})}
Let $A$, $B$ and $C$ be algebras, and $\mathcal{D}A$ admit a
recollement relative to $\mathcal{D}B$ and $\mathcal{D}C$. Then $A$
is smooth if and only if so are $B$ and $C$.
\end{proposition}

However, Proposition~\ref{Proposition-smooth} is not correct for
homological smoothness any more. Here is an example:

\begin{example} \label{Example-Kronecker} {\rm (See \cite[Remark 4]{Han14})
Let $A$ be the infinite Kronecker algebra $\left[\begin{array}{cc} k
& 0 \\ k^{(\mathbb{N})} & k  \end{array}\right]$. Then by
Example~\ref{Example-n-recollement} (2), $\mathcal{D}A$ admits a
2-recollement relative to $\mathcal{D}k$ and $\mathcal{D}k$, but $A$
is not homologically smooth. }
\end{example}

Due to Example~\ref{Example-Kronecker}, even though $\mathcal{D}A$
admits a 2-recollement relative to $\mathcal{D}B$ and
$\mathcal{D}C$, the homological smoothness of $B$ and $C$ can not
imply the homological smoothness of $A$. Nonetheless, we have the
following theorem which is just Theorem II.

\begin{theorem} \label{Theorem-homologically smooth}
Let $A$, $B$ and $C$ be algebras, and $\mathcal{D}A$ admit an
$n$-recollement relative to $\mathcal{D}B$ and $\mathcal{D}C$.

{\rm (1)} $n=1$: if $A$ is homologically smooth then so is $B$;

{\rm (2)} $n=2$: if $A$ is homologically smooth then so are $B$ and
$C$;

{\rm (3)} $n \geq 3$: $A$ is homologically smooth if and only if so
are $B$ and $C$.
\end{theorem}

\begin{proof}
(1) See \cite[Proposition 3.10 (c)]{Kel11}.

(2) If $A$ is homologically smooth then $B$ is also
homologically smooth by (1). Since $n=2$, we have a
recollement of $\mathcal{D}A$ relative to $\mathcal{D}C$ and
$\mathcal{D}B$, and thus $C$ is also homologically smooth by (1)
again.

(3) Assume $\mathcal{D}A$ admits a 3-recollement relative to
$\mathcal{D}B$ and $\mathcal{D}C$. Consider the recollement formed
by the middle three layers. By \cite[Proposition 3]{Han14}, we may
assume that it is of the form
$$\xymatrix@!=6pc{ \mathcal{D}C \ar[r]^{i_* = - \otimes^L_C Y^\star}
& \mathcal{D}A \ar@<-3ex>[l]_{i^* = -\otimes^L_A Y}
\ar@<+3ex>[l]_{i^!} \ar[r]^{j^! = -\otimes _A^L X^*} & \mathcal{D}B
\ar@<-3ex>[l]_{j_! = -\otimes^L_B X} \ar@<+3ex>[l]_{j_*} } \eqno
{\rm (R')}$$ where $X \in \mathcal{D}(B^{\op} \otimes A)$, $Y \in
\mathcal{D}(A^{\op} \otimes C)$, $X^* := \RHom_A(X,A)$ and $Y^\star
:= \RHom_C(Y,C)$. Clearly, $_AY$ and $Y^\star_A$ are compact since
the recollement can be extended one step upwards and one step
downwards respectively \cite[Proposition 3.2 and Lemma 2.8]{AKLY13}.

By \cite[Theorem 1]{Han14} and \cite[Theorem 2]{Han14}, a
recollement of derived categories of algebras induces those of
tensor product algebras and opposite algebras respectively. Thus we
have the following three recollements induced by the recollement
${\rm (R')}$:
$$\xymatrix @R=0.6in @C=0.8in{\mathcal{D}(C^e) \ar[d]|{F_1} & &
\mathcal{D}(C^{\op} \otimes_k B) \ar[d] \\
\mathcal{D}(A^{\op} \otimes_k C) \ar@<+3ex>[u] \ar@<-3ex>[u]|{L_1}
\ar[d] \ar[r]|{F_2} & \mathcal{D}(A^e) \ar@<+3ex>[l]
\ar@<-3ex>[l]|{L_2} \ar[r]|{F_3} & \mathcal{D}(A^{\op} \otimes_k B)
\ar@<+3ex>[u] \ar@<-3ex>[u] \ar@<+3ex>[l] \ar@<-3ex>[l]|{L_3} \ar[d]|{F_4} \\
\mathcal{D}(B^{\op} \otimes_k C) \ar@<+3ex>[u] \ar@<-3ex>[u] & &
\mathcal{D}(B^e) \ar@<+3ex>[u] \ar@<-3ex>[u]|{L_4} }$$ where $L_1 =
Y^\star \otimes_A^L -$, $F_1 = Y \otimes_C^L -$, $L_2 = -
\otimes_A^L Y$, $F_2 = - \otimes_C^L Y^\star$, $L_3 = - \otimes_B^L
X$, $F_3 = - \otimes_A^L X^*$, $L_4 = X^*\otimes_B^L -$ and $F_4 = X
\otimes_A^L -$. Consider the canonical triangle
$$L_3F_3A \longrightarrow A \longrightarrow F_2L_2A \longrightarrow \; \mbox{in} \; D(A^e),$$
and note that $F_2L_2A = Y \otimes_C^L Y^\star = F_2F_1C$, $L_3F_3A
= X^* \otimes_B^L X = L_3L_4B$. Clearly, the functors $L_3$ and
$L_4$ preserve compactness, so are $F_1$ and $F_2$ since $_AY$ and
$Y^\star_A$ are compact. Applying these to the above triangle, we
get $A \in K^b(\proj A^e)$ whenever $B \in K^b(\proj B^e)$ and $C
\in K^b(\proj C^e)$. Namely, the homological smoothness of $B$ and
$C$ implies that of $A$.
\end{proof}

According to Example~\ref{Example-Kronecker} and the statement
followed, we see that in Theorem~\ref{Theorem-homologically smooth}
(3), the requirement $n \geq 3$ is optimal.

Applying Theorem~\ref{Theorem-homologically smooth} to triangular
matrix algebras, we get the following corollary which provides a
construction of homologically smooth algebras.

\begin{corollary}
Let $B$ and $C$ be algebras, $M$ a $C$-$B$-bimodule, and $A :=
\left[\begin{array}{cc} B & 0 \\ M & C  \end{array}\right] $.

{\rm(1)} If $A$ is homologically smooth, then so are $B$ and $C$;

{\rm(2)} If $B$ and $C$ are homologically smooth and $_CM \in
K^b(\proj C^{\op})$ or $M_B \in K^b(\proj B)$, then $A$ is also
homologically smooth.
\end{corollary}

\begin{proof}
It follows from Example~\ref{Example-n-recollement} (2) and
Theorem~\ref{Theorem-homologically smooth}.
\end{proof}

\section{\large $n$-recollements and Gorensteinness}

\indent\indent In this section, we will observe the relations
between $n$-recollements and the Gorensteinness of algebras, and
reduce the Gorenstein symmetry conjecture to 2-derived-simple
algebras.

A finite dimensional algebra $A$ is said to be {\it Gorenstein} if
$\id_A A < \infty$ and $\id_{A^\op}A < \infty$. Clearly, a finite
dimensional algebra $A$ is Gorenstein if and only if $K^b(\proj A) =
K^b(\inj A)$ as strictly full triangulated subcategories of
$\mathcal{D}A$. Thus, the Gorensteinness of algebras is invariant
under derived equivalences. It is natural to consider the relation
between recollements and the Gorensteinness of algebras. In
\cite{Pan13}, Pan proved that the Gorensteinness of $A$ implies the
Gorensteinness of $B$ and $C$ if there exists a recollement of
$\mathcal{D}^b(\mod A)$ relative to $\mathcal{D}^b(\mod B)$ and
$\mathcal{D}^b(\mod C)$. Now we complete it using the language of
$n$-recollements. The following theorem is just Theorem III.

\begin{theorem}\label{Theorem-Gorenstein}
Let $A$, $B$ and $C$ be finite dimensional algebras, and
$\mathcal{D}A$ admit an $n$-recollement relative to $\mathcal{D}B$
and $\mathcal{D}C$.

{\rm (1)} $n=3$: if $A$ is Gorenstein then so are $B$ and $C$;

{\rm (2)} $n \geq 4$: $A$ is Gorenstein if and only if so are $B$ and
$C$.
\end{theorem}

\begin{proof}
(1) It follows from Proposition~\ref{Proposition-3-recollement} that
$\mathcal{D}^b(\mod A)$ admits a recollement relative to
$\mathcal{D}^b(\mod C)$ and $\mathcal{D}^b(\mod B)$. Therefore, the
statement follows from Pan \cite{Pan13}. Here we provide another
proof. Consider the following recollement consisting of the middle
three layers of functors of the 3-recollement:
$$\xymatrix@!=6pc{ \mathcal{D}(C) \ar[r]^{i_*} &
\mathcal{D}(A) \ar@<-3ex>[l]_{i^*} \ar@<+3ex>[l]_{i^!} \ar[r]^{j^!}
& \mathcal{D}(B) \ar@<-3ex>[l]_{j_!} \ar@<+3ex>[l]_{j_*}}.$$ By
Lemma~\ref{Lemma-restrict}, $i^*, i_*, j_!$ and $j^!$ restrict to
$K^b(\proj)$, and $i_*, i^!, j^!$ and $j_*$ restrict to $K^b(\inj)$.
If $A$ is Gorenstein then $K^b(\proj A)= K^b(\inj A)$.

Note that $DC := \Hom _k(C,k)$. Thus $DC \cong i^*i_*(DC) \in
i^*i_*(K^b(\inj C)) \subseteq i^*(K^b(\inj A)) = i^*(K^b(\proj A))
\subseteq  K^b(\proj C).$ Thus, $\pd_C (DC) < \infty$, equivalently,
$\id_{C^\op}C < \infty$. On the other hand, $C \cong i^!i_*C \in
i^!i_*K^b(\proj C) \subseteq i^!K^b(\proj A) = i^!K^b(\inj A)
\subseteq K^b(\inj C)$. Thus, $\id_C C < \infty$. Therefore, $C$ is
Gorenstein.

Similarly, $DB \cong j^!j_*(DB) \in j^!j_*(K^b(\inj B)) \subseteq
j^!(K^b(\inj A)) = \linebreak j^!(K^b(\proj A)) \subseteq  K^b(\proj
B).$ Thus, $\pd_B (DB) < \infty$, equivalently, $\id_{B^\op}B <
\infty$. On the other hand, $B \cong j^!j_!B \in j^!j_!K^b(\proj B)
\subseteq j^!K^b(\proj A) = j^!K^b(\inj A) \subseteq K^b(\inj B)$.
Thus, $\id_B B < \infty$. Therefore, $B$ is Gorenstein.

\medskip

(2) Let $$\xymatrix@!=6pc{ \mathcal{D}B \ar@<2.4ex>[r] \ar@<-4ex>[r]
\ar@<-0.8ex>[r]|{i_*} & \mathcal{D}A \ar@<-4ex>[l]
\ar@<-0.8ex>[l]|{i^*} \ar@<+2.4ex>[l] \ar@<-0.8ex>[r]|{j^!}
\ar@<2.4ex>[r] \ar@<-4ex>[r] & \mathcal{D}C \ar@<-4ex>[l]
\ar@<-0.8ex>[l]|{j_!} \ar@<2.4ex>[l]}$$ be a 4-recollement. By
Lemma~\ref{Lemma-restrict}, $i^*$, $j_!$, $i_*$ and $j^!$ restrict
to both $K^b(\proj)$ and $K^b(\inj)$. If both $B$ and $C$ are
Gorenstein, then $K^b(\proj B) = K^b(\inj B)$ and $K^b(\proj C)=
K^b(\inj C)$.

Consider the triangle $j_!j^!(DA) \rightarrow DA \rightarrow
i_*i^*(DA) \rightarrow .$ We have $j_!j^!(DA) \linebreak \in
j_!j^!K^b(\inj A) \subseteq j_!K^b(\inj C) = j_!K^b(\proj C)
\subseteq K^b(\proj A)$ and $i_*i^*(DA) \in i_*i^*K^b(\inj A)
\subseteq i_*K^b(\inj B) = i_*K^b(\proj B) \subseteq K^b(\proj A)$.
Thus, $DA \in K^b(\proj A)$, i.e., $\pd_A(DA)<\infty$. Hence,
$\id_{A^\op}A < \infty$.

Similarly, consider the triangle $j_!j^!A \rightarrow A \rightarrow
i_*i^*A \rightarrow .$ We have $j_!j^!A \in j_!j^!K^b(\proj A)
\subseteq j_!K^b(\proj C) = j_!K^b(\inj C) \subseteq K^b(\inj A)$
and $i_*i^*A \linebreak \in i_*i^*K^b(\proj A) \subseteq
i_*K^b(\proj B) = i_*K^b(\inj B) \subseteq K^b(\inj A)$. Thus, $A
\in K^b(\inj A)$, i.e., $\id_A A < \infty$. Therefore, $A$ is
Gorenstein.
\end{proof}

Applying Theorem~\ref{Theorem-Gorenstein} to triangular matrix
algebras, we get the following corollaries, which imply the
condition $n \geq 4$ in Theorem~\ref{Theorem-Gorenstein} (2) is
optimal.

\begin{corollary}{\rm (\cite[Theorem 3.3]{Chen09})}
Let $B$ and $C$ be Gorenstein algebras, $M$ a finite generated
$C$-$B$-bimodule, and $A = \left[\begin{array}{cc} B & 0 \\ M & C
\end{array}\right] $. Then $A$ is Gorenstein if and only if $\pd_{C^\op}M
< \infty$ and $\pd_B M < \infty$.
\end{corollary}

\begin{proof} Assume that $A$ is Gorenstein.
Set $e_1 = \left[\begin{array}{cc} 1 & 0 \\ 0 & 0
\end{array}\right]$ and $e_2 = \left[\begin{array}{cc} 0 & 0 \\ 0 & 1
\end{array}\right]$. By \cite[Example 3.4]{AKLY13}, there is a
2-recollement of the form
$$\xymatrix@!=8pc{ \mathcal{D}B \ar@<+1.5ex>[r]|{i_* = -\otimes ^L_Be_1A}
\ar@<-4.5ex>[r] & \mathcal{D}A \ar@<+1.5ex>[r]|{j^!} \ar@<-4.5ex>[r]
\ar@<-4.5ex>[l]|{i^*} \ar@<+1.5ex>[l]|{i^!=-\otimes ^L_AAe_1} &
\mathcal{D}C\ar@<-4.5ex>[l]|{j_!=-\otimes ^L_Ce_2A}
\ar@<+1.5ex>[l]|{j_*}}.$$ It follows from Lemma~\ref{Lemma-restrict}
that $i^!$ restricts to $K^b(\inj)$, and further restricts to
$K^b(\proj)$ by the Gorensteinness of $A$ and $B$. Thus, $i^!A = B
\oplus M_B \in K^b(\proj B)$. Hence, $\pd_B M < \infty$. Similarly,
it follows from Lemma~\ref{Lemma-restrict} that $j^!$ restricts to
$K^b(\proj)$, and further restricts to $K^b(\inj)$ by the
Gorensteinness of $A$ and $C$. By Lemma~\ref{Lemma-adjoint}, $j_!$
restricts to $\mathcal{D}^b(\mod)$. Since $j_! = - \otimes _C ^L
e_2A$, this is equivalent to $_C(e_2A) \in K^b(\proj C^{\op})$ (Ref.
\cite[Lemma 2.8]{AKLY13}). Note that $_C(e_2A) = C \oplus {_CM}$,
thus $\pd_{C^\op}M < \infty$.

Conversely, assume that $\pd_{C^\op}M < \infty$ and $\pd_B M <
\infty$, then by Example~\ref{Example-n-recollement} (2), the above
2-recollement can be extended one step upwards and one step
downwards to a 4-recollement. Therefore, the Gorensteinness of $B$
and $C$ implies the Gorensteinness of $A$ by
Theorem~\ref{Theorem-Gorenstein}.
\end{proof}

\begin{corollary}{\rm (\cite[Theorem 2.2 (iii)]{XZ12})}
Let $B$ and $C$ be finite dimensional algebras, $M$ a finite
generated $C$-$B$-bimodule with $\pd_{C^\op}M < \infty$ and $\pd_B M
<
\infty$, and $A = \left[\begin{array}{cc} B & 0 \\
M & C \end{array}\right]$. Then $A$ is Gorenstein if and only if so
are $B$ and $C$.
\end{corollary}

\begin{proof}
It follows from Example~\ref{Example-n-recollement} (2) and
Theorem~\ref{Theorem-Gorenstein}.
\end{proof}

Next we study the Gorenstein symmetry conjecture.

\medskip

\noindent{\bf Gorenstein symmetry conjecture.} {\it Let $A$ be an
artin algebra. Then $\id_AA < \infty$ if and only if $\id_{A^\op} A
< \infty$.}

\medskip

This conjecture is listed in Auslander-Reiten-Smal{\o}'s book
\cite[p.410, Conjecture (13)]{ARS95}, and it closely connects with
other homological conjectures. For example, it is known that the
finitistic dimension conjecture implies the Gorenstein symmetry
conjecture. But so far all these conjectures are still open. As
mentioned before, the finitistic dimension conjecture can be reduced
to 2-derived-simple algebras. Now, let us utilize
Theorem~\ref{Theorem-Gorenstein} to reduce the Gorenstein symmetry
conjecture to 2-derived-simple algebras.

\begin{proposition} \label{Proposition-GorSymConj}
Let $A$, $B$ and $C$ be finite dimensional algebras, and
$\mathcal{D}A$ admit a $2$-recollement relative to $\mathcal{D}B$
and $\mathcal{D}C$. If both $B$ and $C$ satisfy the Gorenstein symmetry
conjecture, then so does $A$.
\end{proposition}

\begin{proof}
Assume that
$$\xymatrix@!=5pc{ \mathcal{D}B \ar@<+1ex>[r]|{i_*} \ar@<-3ex>[r] & \mathcal{D}A
\ar@<+1ex>[r]|{j^!} \ar@<-3ex>[r] \ar@<-3ex>[l] \ar@<+1ex>[l]|{i^!}
& \mathcal{D}C \ar@<-3ex>[l] \ar@<+1ex>[l]|{j_*} } \eqno {\rm
(R'')}$$ is a 2-recollement, and both $B$ and $C$ satisfy the
Gorenstein symmetry conjecture.

If $\id_A A < \infty$, then $K^b(\proj A) \subseteq K^b(\inj A)$. By
Lemma~\ref{Lemma-restrict}, we have $B \cong i^!i_*B \in
i^!i_*(K^b(\proj B)) \subseteq i^!(K^b(\proj A)) \subseteq
i^!(K^b(\inj A)) \subseteq K^b(\inj B)$, i.e., $\id _B B < \infty$.
Since $B$ satisfies the Gorenstein symmetry conjecture, we obtain
that $B$ is Gorenstein. By Lemma~\ref{Lemma-restrict} again, we have
$i^!A \in i^!(K^b(\proj A))\subseteq i^!(K^b(\inj A)) \subseteq
K^b(\inj B)=K^b(\proj B)$. Due to Lemma~\ref{Lemma-downwards}, $i^!A
\in K^b(\proj B)$ implies that the 2-recollement ${\rm (R'')}$ can
be extended one step downwards. Therefore, we get a 2-recollement of
$\mathcal{D}A$ relative to $\mathcal{D}C$ and $\mathcal{D}B$.
Analogous to the above proof, we obtain that $C$ is Gorenstein and
the 2-recollement ${\rm (R'')}$ can be extended two steps downwards
to a 4-recollement of $\mathcal{D}A$ relative to $\mathcal{D}B$ and
$\mathcal{D}C$. By Theorem~\ref{Theorem-Gorenstein}, $A$ is
Gorenstein. Thus $\id_{A^\op}A < \infty$.

If $\id_{A^\op} A < \infty$, then $K^b(\inj A) \subseteq K^b(\proj
A)$. By Lemma~\ref{Lemma-restrict}, we have $DC \cong j^!j_*(DC) \in
j^!j_*(K^b(\inj C)) \subseteq j^!(K^b(\inj A)) \subseteq
j^!(K^b(\proj A)) \subseteq K^b(\proj C)$, i.e., $\id _{C^\op} C <
\infty$. Since $C$ satisfies the Gorenstein symmetry conjecture, we
obtain that $C$ is Gorenstein. By Lemma~\ref{Lemma-restrict} again,
we have $j^!(DA) \in j^!(K^b(\inj A)) \subseteq j^!(K^b(\proj A))
\subseteq K^b(\proj C) = K^b(\inj C)$. Due to
Lemma~\ref{Lemma-upwards}, $j^!(DA) \in K^b(\inj C)$ implies that
the 2-recollement ${\rm (R'')}$ can be extended one step upwards.
Therefore, we get a 2-recollement of $\mathcal{D}A$ relative to
$\mathcal{D}C$ and $\mathcal{D}B$. Analogous to the above proof, we
obtain that $B$ is Gorenstein and the 2-recollement ${\rm (R'')}$
can be extended two steps upwards to a 4-recollement of
$\mathcal{D}A$ relative to $\mathcal{D}B$ and $\mathcal{D}C$. By
Theorem~\ref{Theorem-Gorenstein}, $A$ is Gorenstein. Thus $\id_AA <
\infty$.
\end{proof}

The following corollary implies that the Gorenstein symmetry
conjecture can be reduced to an arbitrary complete set of
representatives of the derived equivalence classes of finite
dimensional algebras.

\begin{corollary} Let $A$ and $B$ be derived equivalent finite
dimensional algebras. Then $A$ satisfies the Gorenstein symmetry
conjecture if and only if so does $B$. \end{corollary}

\begin{proof} It is enough to apply
Proposition~\ref{Proposition-GorSymConj} to the trivial
2-recollements, see Example~\ref{Example-n-recollement} (4).
\end{proof}

Applying Proposition~\ref{Proposition-GorSymConj}, we can reduce the
Gorenstein symmetry conjecture to 2-derived-simple algebras.

\begin{corollary}
The Gorenstein symmetry conjecture holds for all finite dimensional algebras
if and only if it holds for all {\rm 2}-derived-simple algebras.
\end{corollary}

\begin{proof} For any finite dimensional algebra $A$,
by \cite[Proposition 6.5]{AKLY13}, $\mathcal{D}A$ admits a finite
stratification of derived categories along 2-recollements with
2-derived-simple factors. Then the corollary follows from
Proposition~\ref{Proposition-GorSymConj}.
\end{proof}

\bigskip

\noindent {\footnotesize {\bf ACKNOWLEDGMENT.} The authors are very
grateful to Shiping Liu, Baolin Xiong and Dong Yang for many helpful
discussions and suggestions. They thank Peter J{\o}rgensen and
Manuel Saor\'{\i}n for their valuable comments and remarks. They are
sponsored by Project 11171325 NSFC.}

\end{document}